\documentclass[]{article}

\usepackage{amsmath,amssymb,amsfonts,amsthm}
\usepackage{graphicx,color}
\usepackage{enumerate}
\usepackage{subfigure,psfrag}
\graphicspath{{./Figures/}}

\newcommand{\until}[1]{\{1,\dots, #1\}}
\newcommand{\fromto}[2]{\{#1,\dots, #2\}}
\newcommand{\subscr}[2]{#1_{\textup{#2}}}

\newcommand{\setdef}[2]{\left\{#1 \, : \; #2\right\}}
\newcommand{\map}[3]{#1: #2 \rightarrow #3}

\newcommand{\union}{\operatorname{\cup}}
\newcommand{\intersection}{\ensuremath{\operatorname{\cap}}}

\newcommand{\real}{\mathbb{R}}

{
\theoremstyle{definition}
\newtheorem{example}{Example}
\newtheorem{remark}{Remark}
}
\newtheorem{theorem}{Theorem}

\newtheorem{definition}{Definition}

\newtheorem{lemma}[theorem]{Lemma}
\newtheorem{corollary}[theorem]{Corollary}
\newtheorem{proposition}[theorem]{Proposition}

\newcommand{\be}{\begin{equation}}
\newcommand{\ee}{\end{equation}}

\renewcommand{\l}{\left}
\renewcommand{\r}{\right}

\newcommand{\K}{\mathcal{K}} 	
\newcommand{\I}{\mathcal{I}}		
\newcommand{\N}{\mathcal{N}}	
\newcommand{\G}{\mathcal{G}}	
\newcommand{\E}{\mathcal{E}}	
\newcommand{\co}{\overline{\textup{co}}} 
\newcommand{\xave}{\subscr{x}{ave}}	
\newcommand{\xmin}{\subscr{x}{min}}	
\newcommand{\xmax}{\subscr{x}{max}}	
\newcommand{\card}[1]{{\left|{#1}\right|}}	

\newcommand{\car}{Carath\'eodory\ }		

\newcommand{\1}{\mathbf{1}}

\title{Continuous-time Discontinuous Equations in Bounded Confidence Opinion Dynamics}

\author{Francesca Ceragioli \and Paolo Frasca\thanks{F. Ceragioli and P. Frasca are with the Dipartimento di Matematica, Politecnico di Torino, Torino, Italy. The authors wish to thank J. Hendrickx for his comments about this work.}}

\date{}

\begin{document}

\maketitle

\begin{abstract}
This report studies a continuous-time version of the well-known Hegselmann-Krause model of opinion dynamics with bounded confidence. 
As the equations of this model have discontinuous right-hand side, we study their Krasovskii solutions.
We present results about existence and completeness of solutions, and asymptotical convergence to equilibria featuring a ``clusterization'' of opinions. The robustness of such equilibria to small perturbations is also studied.
\end{abstract}


\section{Introduction and Preliminaries}
In the study of social dynamics, which has recently attracted much interest from physicists, mathematicians and control theorists, modeling the interactions between individuals is the core issue. In many successful models of opinion evolution, interactions are assumed to take place only if the opinions of the interacting agents are close enough, say, closer than a certain threshold: such models are usually called ``bounded confidence'' models. In this work, we study a continuous-time version, introduced in~\cite{VDB-JMH-JNT:09a}, of the ``Hegselmann-Krause'' model. 

\subsection*{Problem Statement and Contribution}
Let us consider a population of $N$ agents, indexed in a set $\I=\until{N}$. Each of them has a time-dependent real-valued ``opinion'' $x_i(t)$,  which obeys the following dynamics
\be\label{eq:DHK}
\dot x_i=\sum_{j\in \I} s(x_j-x_i)(x_j-x_i), \qquad i\in\I,
\ee
where $s: \real\to \real$ is defined by
$$s(\tau)=\begin{cases} 1 & {\rm if}\   |\tau|<1\cr
0& {\rm if} \  |\tau|\geq 1. 
\end{cases}
$$
Notice that in the above model agent $j$ influences agent $i$ only if $|x_i-x_j|<1$, and that 
the function $s$, which encodes the coupling between the agent opinions and the bounded confidence assumption, is discontinuous.
Then, the system's right-hand side is discontinuous, and this requires to define solutions in suitable sense: in this paper, we use Krasovskii solutions, which are defined in the next section. Using Krasovskii solutions to solve a models of opinion dynamics with bounded confidence is the main novelty of our work.

On one hand, the advantage of using Krasovskii solutions is technical, as a complete existence and Lyapunov theory is available for such solutions, and this theory allows us to obtain interesting results without making the analysis cumbersome. Moreover the set of  Krasovskii solutions is ``large'', so that the results that we state for Krasovskii solutions also hold for other types of solutions as Filippov solutions and Carath\'eodory solutions, in case they exist.
On the other hand, the convexification of the discontinuities, which is the main feature of Krasovskii's definition, can also be seen as a refinement of the original model, which smooths out sharp decision thresholds. 

Besides the novelty in the solutions' definition, a distinctive feature of this report is our focus on the fundamental properties of the model, in particular (a) average preservation, (b) order preservation, (c) contractivity, (d) existence of Lyapunov functions, (e) robustness of equilibria with respect to small perturbations. In the literature about  networked systems, these properties often play an important role in coordination problems: in our work, they are instrumental to derive the significant results about solutions, namely existence, completeness,  and convergence to a clustered configuration. By clustered configuration we mean a state configuration such that, for every pair of agents, either their individual states coincide, or they differ by more than 1.

\subsection*{Related Works}
Recently, a vast literature (cf. the surveys~\cite{CC-SF-VL:09,JL:07}) has been produced about the evolution of opinions in social dynamics, and the interaction rules which shape such evolution.
Scholars have remarked that, in spite of significant social forces towards homogenization and consensus,
disagreement persists in societies. A convincing explanation for this phenomenon is based on 
``bounded confidence'': interactions are limited to be effective only between individuals whose opinions are already close enough.
A celebrated bounded confidence model is due to Hegselmann and Krause~\cite{RH-UK:02}. Similarly to most other models, the model consists of a discrete-time dynamical system:
discrete time allows for immediate computer simulations and other analysis advantages,
which are exploited for instance in~\cite{VDB-JMH-JNT:09}, but it entails the drawback of assuming synchrony among the opinions updates. 
For this reason, it seems worth to consider continuous-time bounded confidence models, in which a fixed time schedule is not required. This idea was already developed in~\cite{VDB-JMH-JNT:09a}, and indeed our work is strictly related to the one presented in~\cite[Section~2]{VDB-JMH-JNT:09a} and in the auxiliary report~\cite{VDB-JMH-JNT:09b}. In that pair of papers, the authors study the same bounded confidence model as here, but they assume a more restrictive (stronger) definition of solutions: this work provides an extension of their results. 
The relationship of our results with the mentioned pair of papers is further discussed in Section~\ref{sec:compare}.
%

Finally, we note that a conference version of this work has appeared in the Proceedings of the 18th World Congress of the International Federation of Automatic Control as~\cite{FC-PF:10} and that a related paper focusing on the role of the discontinuity of $s$ is going to appear as~\cite{FC-PF:11}.

\subsection*{Notation and Preliminaries}
\subsubsection*{Graphs} 
In this paper, we shall make use of some notions from graph theory, and in particular from algebraic graph theory.
Indeed, graph theory provides an effective tool to model interactions between agents and its use is becoming common both in engineering~\cite{FB-JC-SM:09,MM-ME:10} and in economics and social sciences~\cite{DE-JK:10}. 
A (weighted) graph $G$ is a triple $(V,E,A)$ where $V$ is a finite set
of vertices or nodes, $E\subset V\times V$ is a set of edges and the adjacency matrix $A$ is a matrix of weights, such that for any $u,v\in V$, $A_{uv}>0$ only if $(u,v)\in E$. The Laplacian matrix of $G$ is defined as $L_{uv}=-A_{uv}$ when $u\neq v$ and $L_{uu}=\sum_{v\in V}{A_{uv}}.$ If $(u,v)\in E$, then $v$ is said to be a neighbor of $u$ in $G$.
A path (of length $l$) from $u$ to $v$ in $G$ is an ordered list of edges $(e_1, \ldots, e_l)$ 
in the form $((u, w_1),(w_1,w_2), (w_2,w_3), \ldots, (w_{l-1},v)).$
Two nodes $u,v\in V$ are said to be connected if there exists a path
from $u$ to $v$, and disconnected otherwise. A graph is said to be
connected if every two nodes are connected, and disconnected
otherwise.  A graph is said to be symmetric when $(u,v)\in E$ implies
$(v,u)\in E$ and the matrix $A$ is symmetric. In a symmetric graph, being neighbors is an equivalence relation between nodes: the corresponding equivalence classes are said to be the connected components of the graph.

\subsubsection*{Solutions to ODEs}
As already remarked, in order to deal with possibly discontinuous
ODEs,  we need to take different notions of solutions
into consideration. 
We provide the definitions of classical, \car and Krasovskii solutions: the interested reader can find more background information in classical books as~\cite{AFF:88} or in the recent tutorial~\cite{JC:08-csm}. 
Let us consider the  differential equation
\be\label{eq:g}
\begin{cases}
\dot x=g(x)\\
x(t_0)=\bar x
\end{cases}
\ee
where $\map{x}{\real}{\real^N}$,  $\map{g}{\real ^N}{\real ^N}$, and $\bar x\in\real^N$.

A {\em classical solution} to  \eqref{eq:g} on an  interval $I\subset \real$ containing $t_0$, is a map $\map{\phi}{I}{\real^N}$ such that
\begin{enumerate}
\item $\phi$ is differentiable  on $I$,
\item $\phi (t_0)=\bar x$,
\item  $\dot{\phi}(t) =g(\phi(t) )$ for all  $t\in I$.
\end{enumerate}
 
A {\em Carath\'eodory solution} to \eqref{eq:g}  on an  interval $I\subset \real$ containing $t_0$, is a map $\map{\phi}{I}{\real^N}$ such that
\begin{enumerate}
\item $\phi$ is absolutely continuous on $I$,
\item $\phi (t_0)=\bar x$,
\item  $\dot{\phi}(t) =g(\phi (t))$ for almost every $t\in I$.
\end{enumerate}
 Equivalently, a Carath\'eodory solution to  \eqref{eq:g} is a solution to the integral equation
$$x(t)=\bar x+\int _{t_0}^{t}g(x(s))ds.  $$

A {\em Krasovskii solution} to \eqref{eq:g} on an  interval $I\subset \real$ containing $t_0$, is a map $\map{\phi}{I}{\real^N}$ such that
\begin{enumerate}
\item $\phi$ is absolutely continuous on $I$,
\item $\phi (t_0)=\bar x$, 
\item $\dot \phi(t)\in \K g(\phi(t))$ for almost every $t\in I$,
where $$\K g(x) =\bigcap_{\delta>0}\co(\setdef{g(y)}{y \text{ such
    that } \|x-y\|<\delta})$$ 
and given a set $A$, by $\co(A)$ we denote the closed convex hull of $A$.
\end{enumerate}

From the above definitions, it is clear that classical solutions are
\car solutions and, in turns, \car solutions are Krasovskii
solutions. Note also that \car solutions coincide with solutions in the classical sense when $g$ is continuous.

\section{Analysis and Results}
\label{sec:DHK}

This section contains our results about the dynamics~\eqref{eq:DHK}. After providing a graph-theoretical interpretation of the system, we prove several preliminary properties of Krasowskii solutions, including completeness and order preservation. Then, we take advantage of these results to prove convergence to clustered configurations. Finally, we define the robustness of clusters to small perturbations, and provide a necessary and sufficient condition.

\subsection{Interaction Graphs and Basic Properties}\label{sec:useful-graphs}
It is useful and suggestive to rewrite system~\eqref{eq:DHK} as a dynamic over a suitable state-dependent weighted graph, which represents the coupling between the opinions of different agents. 
In such a  graph the agents are the nodes, and the opinions of two agents  depend on each other whenever the agents are neighbors in the graph. 
By the way system~\eqref{eq:DHK} is defined, such interaction graph depends on the opinion states via the function $s$.
More precisely, for any $x\in \real ^N$ we define an interaction graph
$\G (x)=(\I, \E(x), A(x))$ where the edge set is
$$
\E(x)=\setdef{(i,j),\, i,j \in \I}{|x_i-x_j|<1},$$ that is,
$(i,j)\in \E(x)$ if and only if 
$s(x_j-x_i)>0$, 
and the adjacency matrix $A(x)$ is defined by
$$
A(x )_{ij}=\begin{cases}  s(x_i-x_j) & {\rm if}\  j\not =i\cr
0 &  {\rm if}\  j=i
\end{cases}
\qquad i,j\in \I,
$$ that is, $A(x)_{ij}=1$ if and only if $|x_i-x_j|<1$ and $j\neq i$.
The Laplacian matrix $L(x)$ associated to $\G (x)$ is then given by
$$
L(x )_{ij}=\begin{cases}  -s(x_i-x_j) & {\rm if}\  j\not =i\cr
\sum _{k\not = i}s(x_k-x_i) &  {\rm if}\  j=i
\end{cases}
\qquad i,j\in \I .
$$

In order to deal with the discontinuity, it is also useful to identify
``border'' configurations by the following definitions of border edge set 
$$\partial\E(x)=\setdef{(i,j),\, i,j \in \I}{|x_i-x_j|=1}.$$
and graph $\bar\G(x)=\left( \I, \bar\E(x),\bar A\right),$ with $\bar\E(x)=\E(x)\union\partial\E(x)$ and $\bar A(x)_{ij}=1$ if and only if $|x_i-x_j|\le1$ and $j\neq i$.

\begin{remark}[Symmetry and translation invariance]\label{rem:trans-invar}
We remark that the graphs $\G$ and $\bar \G$ are symmetric and invariant with respect to the translation $x+\alpha\1$, where $\alpha \in \real$ and $\1=(1,...,1)^T$, i.e.  
$\G(x)=\G( x+\alpha\1)$ and $\bar \G(x)=\bar \G( x+\alpha\1).$ \end{remark}

As we said, the graphs introduced above are interaction graphs in the
following sense: if two nodes are disconnected, they can not influence
each other opinions.

\smallskip
With the above notation, system~\eqref{eq:DHK} can be written as
\begin{equation}\label{eq:f(x)}
\dot x=-L(x)x,\end{equation}
being 
$L(x)$ the Laplacian matrix of the state-dependent graph $\G (x)$ and  
$$(-L(x)x)_i=\sum_{j\in \I} s(x_j-x_i)(x_j-x_i)$$
the components of the right-hand side.
As the differential equation~\eqref{eq:f(x)} has a discontinuous right-hand side, we consider Krasovskii solutions to~\eqref{eq:DHK}, which we characterize as follows.
For any $H\subset \partial\E(x)$ we let $L^H(x)$ be the Laplacian matrix associated to the graph $\G ^H (x)$ with edges  $\E (x)\union H$, and correspondingly 
$$(-L^H(x)x)_i=\sum_{j: (i,j)\in\E (x)\union H } (x_j(t)-x_i(t)).$$

By the definition, it is clear that a Krasovskii solution to~\eqref{eq:DHK} satisfies at almost every time the inclusion $$ \dot x \in \co(\setdef{-L^H(x)x}{H\subset \partial\E(x)}),$$
or equivalently the inclusion
\begin{align*}
 \dot x
 \in \Big\{
-\sum_{H\subset \partial\E(x)}\!\!  & \alpha_H L^H(x)x
 \;:\; 
\alpha_H\ge0, \: 
  \sum_{K\subset\partial\E(x)}\!\!\alpha_K=1
 \Big\}.
 \end{align*}
Namely, for a given Krasovskii solution $\phi(\cdot)$,
$$ \dot\phi(t) 
=-\sum_{H\subset \partial\E(\phi(t))} \alpha_H^\phi(t) L^H(\phi(t))\phi(t)
 \quad \text{ for almost every $t$},$$ 
where the time-dependent coefficients $\alpha_H^\phi$ depend on the
solution $\phi(\cdot)$ itself.

\bigskip
Using this graph-theoretical characterization, we now prove some basic properties of Krasovskii solutions to~\eqref{eq:DHK}.

\begin{proposition}[Basic properties of solutions]\label{prop:DHK}
Let $x(\cdot)$ be a Krasovskii solution to~\eqref{eq:DHK}, on its domain of definition. 
\begin{enumerate}[(i)]
\item\label{item:prop-disc:exist}{\rm (Existence).}  For any initial condition $\bar x\in \real^\I$, there exists a local Krasovskii solution to~\eqref{eq:DHK}.
\item\label{item:prop-disc:order}{\rm (Order preservation).}
For any $i,j\in\I,$ if $x_i(t_1)<x_j(t_1)$, then $x_i(t_2)<x_j(t_2)$, for any $t_2>t_1$.
\item\label{item:prop-disc:contractive}{\rm (Contractivity).}
For any $t_2>t_1,$ $\co(\{x_i(t_2)\}_{i\in\I})\subset\co(\{x_i(t_1)\}_{i\in\I})$.
\item\label{item:prop-disc:complete}{\rm (Completeness).} The solution $x (\cdot)$ is complete.
\item\label{item:prop-disc:average}{\rm (Average preservation).} Let $x_{\rm ave}(t)=N^{-1}\sum_{i=1}^{N} x_i(t)$. Then $x_{\rm ave}(t)=x_{\rm ave}(0)$, for $t>0$; 
\end{enumerate}
\end{proposition}

\begin{proof}
In the proof, the following notation will be useful.
For every $i\in\I$, and every $x\in \real^N$, we let $$\N_i(x):=\setdef{k\in\I}{|x_i-x_k|<1},$$
and for any $H\subset \partial\E(x)$, we let 
$$N^{H}_i=\setdef{k\in\I}{(i,k)\in H}.$$
Clearly,  $N^{H}_i \subset \partial \N_i(x):=\setdef{k\in\I}{|x_i-x_k|=1}.$
With this notation, $$(-L(x)x)_i=\sum_{j\in \N_i(x)} (x_j(t)-x_i(t))$$
and 
$$(-L^H(x)x)_i=\sum_{j\in \N_i(x) \union N^H_i} (x_j(t)-x_i(t)).$$
We are now ready to prove our statements.
\begin{enumerate}[i)]
\item Since the right-hand side of~\eqref{eq:DHK} is locally essentially bounded, local existence of a Krasovskii solution is guaranteed (see for instance~\cite{OH:79}). 
\item 
To prove the claim, we study the dynamics of the difference between $x_j$ and $x_i$. By continuity of the solutions, we can assume with no loss of generality that $x_j$ and $x_i$ are close, for instance that $x_j-x_i<1$. For brevity, in the following we omit the explicit dependence of $\N_i$ on $x$. 
For almost every time $t$, we have 
\begin{align*}
\frac{d}{dt}& (x_j-x_i)
= 
\sum_{H\in\partial\E(x)} \alpha_H \left[
\sum_{h\in\N_j\union N^H_j}(x_h-x_j)- \sum_{h\in\N_i\union N^H_i}(x_h-x_i) \right]\\
=&
\sum_{H\in\partial\E(x)} \alpha_H \left[
\sum_{h\in\N_i\intersection\N_j}((x_h-x_j)-(x_h-x_i)) 
\right. 
+ \sum_{h\in\N_i\intersection N_j^H}((x_h-x_j)-(x_h-x_i))
 \\ 
&\qquad \qquad  +
\sum_{h \in N_i^H\intersection\N_j}((x_h-x_j)-(x_h-x_i))
- \sum_{h\in(\N_i\union N^H_i)\setminus(\N_j\union N^H_j)} (x_h-x_i)
  \\
&\qquad \qquad  
+ \left. \sum_{h\in(\N_j\union N^H_j)\setminus(\N_i\union N^H_i)} (x_h-x_j)
\right]\\ 
=&
\sum_{H\in\partial\E(x)} \alpha_H 
\left[
\sum_{h\in\N_i\intersection\N_j}-(x_j-x_i) +
\sum_{h\in\N_i\intersection N_j^H} (-x_j+x_i) 
\right.\left. 
+\sum_{h \in N_i^H\intersection\N_j}(-x_j+x_i)
\right.\\ 
&\qquad \qquad\left.
- \sum_{h\in(\N_i\union N^H_i)\setminus(\N_j\union N^H_j)} (x_h-x_i)
\right. \left. 
+ \sum_{h\in(\N_j\union N^H_j)\setminus(\N_i\union N^H_i)} (x_h-x_j)
\right]\\
=&
-\card{\N_i\intersection\N_j} (x_j-x_i)  
+ \sum_{H\subset \partial\E(x)} \alpha_H
\Big[
- ( \card{\N_i\intersection N_j^H} +\card{N_i^H\intersection\N_j})
 (x_j-x_i)\\
& \qquad \qquad - \sum_{h\in(\N_i\union N^H_i)\setminus(\N_j\union N^H_j)} (x_h-x_i)
+ \sum_{h\in(\N_j\union N^H_j)\setminus(\N_i\union N^H_i)} (x_h-x_j)
\Big] .
\end{align*}
Since if $h\in (\N_i\union N^H_i)\setminus(N_j^H\union \N_j)$, then $x_h-x_i<0$, whereas if $h\in (\N_j\union N^H_j)\setminus(N_i^H\union \N_i)$, then $x_h-x_j>0$, and since $\card{\N_i\intersection N_j^H}\le \card{\N_i},$
we get that 
$$
\frac{d}{dt} (x_j-x_i)\ge - \big(\card{\N_i\intersection\N_j} + \card{\N_i} +\card{\N_j}\big) (x_j-x_i).
$$

The obtained inequality ensures that $x_j-x_i$ can not reach zero in finite time, and yields our claim.

\item 
To prove the claim we show that the leftmost agent can only move to its right. 
To this goal, we need a recall the proof of statement~(\ref{item:prop-disc:order}). While our argument shows that strict inequalities between agents' states are preserved by the dynamics, we have to remark that equalities are not.
It is not in general true that if $x_i(t_1)=x_j(t_1)$, then $x_i(t_2)=x_j(t_2)$ for any $t_2>t_1.$ Indeed, we can observe that if $x_i(t_1)=x_j(t_1)$, then $\dot x_i(t_1)$ and $\dot x_j(t_1)$ have to satisfy to the same differential inclusion, but need not to be equal.
However, it can be proven that it is always possible, given a solution $x(\cdot)$, to sort the states so that $x_{i_1}(t)\le x_{i_2}(t)\le \ldots \le x_{i_N}(t),$ for every $t$. Note that this mapping$\,\map{i_{(\cdot)}}{\until{N}}{\I}$ depends on the solution and needs not to be unique. Nevertheless, it allows us to define $\xmin(t):=x_{i_1}(t)$, and $\xmax(t):=x_{i_N}(t)$. This fact is useful because it allows us to observe that $$x_i(t)-\xmin(t)\ge0$$ for every $t$ and every $i\in\I$ and then, for almost every time $t$,
$\frac{d}{dt}{\xmin}(t)\in[0,+\infty).$ Repeating an analogous argument for $\xmax$ implies the claim.

\item 
Claim~(\ref{item:prop-disc:contractive}) ensures that solutions are bounded. By standard arguments, this is enough to guarantee that local solutions can be extended for all $t>0$.

\item  For every  $x\in \real^N$, every $H\subset\partial\E(x)$ and every $i,j\in\I$, it holds that $j\in\N_i(x)\union N_i^H$ if and only if $i\in\N_j(x)\union N_j^H$; that is, the graph  $\G ^H (x)$ is symmetric. This key remark allows us to argue that for almost every time $t$,
\begin{align*}
\frac d{dt} \xave(t)=&N^{-1}\sum_{i\in\I}\dot x_i(t)\\=&
N^{-1}\sum_{i\in\I}\sum_{H\subset \partial \E(x(t))} \alpha_H(t)\sum_{j\in\N_i(x) \union N^H_i} (x_j(t)-x_i(t))\\
=&N^{-1}\sum_{H\subset \partial \E(x(t))} \alpha_H(t)\sum_{i\in\I}\sum_{j\in\N_i(x) \union N^H_i} (x_j(t)-x_i(t))
=\,0.
\end{align*}
This ensures $\xave(t)=\xave(0)$ for every $t>0.$
\end{enumerate}
\end{proof}

\subsection{Convergence}
We are now ready to prove convergence to a configuration in which agents are separated into clusters of agents which share the same opinion.
We first recall that a point $\tilde x$ is said to be a Krasovskii
equilibrium of \eqref{eq:DHK} if the function $x(t)\equiv \tilde x$ is
a Krasovskii solution to \eqref{eq:DHK}, i.e. $0\in
\co(\setdef{-L^H(\tilde x)\tilde x}{H\subset \partial\E(\tilde x)})$.

\begin{theorem}[Convergence of DHK]\label{th:DHK-2}
The set of Krasovskii equilibria of \eqref{eq:DHK} is 
\begin{align*}
F=\Big\{x\in\real^N\;:\; &\text{for every }  (i,j)\in\I\times\I, \text{ either } x_i=x_j \text{ or } |x_i-x_j|\ge 1\Big\}\end{align*}
and if $x(\cdot)$ is a Krasovskii solution to~\eqref{eq:DHK}, then
$x(t)$ converges to a point $x_*\in F$ as $t\to +\infty$. 
\end{theorem}

\begin{proof}
The proof is in three steps: we first describe the set of equilibria, then prove convergence to this set, and finally prove convergence to one equilibrium.
\begin{enumerate}[i)]
\item 
It is clear that every point in $F$ is an equilibrium. To prove that
there are no other equilibria, we proceed as follows. Without loss of generality we can sort the components of $\tilde x$ so that $\tilde x_{i_1}\leq ...\leq \tilde x_{i_N}$.
For a vector $v\in\co(\setdef{-L^H(\tilde x)\tilde x}{H\subset\partial\E(\tilde x)})$ to be equal to zero, it is necessary that $v_{i_1}=0$. But since $\tilde x_k-\tilde x_{i_1}\ge0$ for every $k\in\I$, it is necessary that $\tilde x_j-\tilde x_{i_1}\in \{0\}\union [1,+\infty),$ for every $j\in \I.$ Repeating this reasoning for $i_2,\ldots$, we have that the set of equilibria actually coincides with $F$.

\item We define the Lyapunov function 
$V(x)=\frac12\sum_{i\in\I}x_i^2$ and compute, using the symmetry of the graph $\G(x)$ as done in~\cite{CC-FF-PT:08},
\begin{align*}
\frac{d}{dt} V(x(t))=&   \sum_{i\in \I} x_i(t) \dot x_i(t)\\
= &  \sum_{i\in \I} x_i(t) \left( \sum_{j\in \N_i(x)} (x_j(t)-x_i(t))+
\sum_{H\subset\partial\E(x)} \alpha_H \sum_{j\in N_i^H} (x_j(t)-x_i(t))\right)\\
=&- \frac12\sum_{(i,j)\in\E(x)} (x_j(t)-x_i(t))^2 - \frac12
\sum_{H\subset\partial\E(x)} \alpha_H \sum_{(i,j)\in H}(x_j(t)-x_i(t))^2 \le 0.
\end{align*}
Since  the inequality is strict if $x(t)\not\in F$, and $F$ is closed and weakly invariant, we can apply a LaSalle invariance principle~\cite[Theorem~3]{AB-FC:99} to conclude convergence to the set~$F$.

\item We observe that the set $F$ is the union of a finite number of sets $F_P$, where $P=\fromto{P_1}{P_k}$ is a partition of $\I$ in $1\le k\le N$ subsets, and 
$$F_P=\setdef{x\in\real^N}{\forall\, i,j\in \I, \text{ if $\exists\, h$ s.t. }  i,j \in P_h, \text{then } x_i=x_j, \text{else } |x_i-x_j|\ge1 }.$$
As the sets $F_P\subset F$ are closed and disjoint, each solution converges towards one of them. 
Without loss of generality, we relabel the states so that, for the solution at hand, $x_1(t)\le \ldots \le x_N(t)$ for every $t\ge0$.
%
When $k=1$, the only partition is the trivial one, corresponding to equilibria in which the states of all agents coincide. In this case, average preservation implies that $x_i(t)\to \xave(0)$ for all $i$ as $t\to\infty.$
When $k=2$, there exists $a\in \until{N}$ such that for every $x\in F_P$ it holds that $x_i=x_a$ for every $i\le a$,  $x_i=x_{a+1}$ for every $i> a$, and $x_{a+1}-x_{a}\ge1$. 
Let $T_a=\inf\setdef{t\ge0}{x_{a+1}(t)-x_a(t)>1}.$ If $T_a<+\infty$, then there is disconnection at finite time and $x_i(t)\to \frac1a \sum_{j\le a} x_j(T)$ if $i\le a$ whereas $x_i(t)\to \frac1{N-a} \sum_{j>a} x_j(T)$ otherwise.
If instead $T=\infty$, then $x_{a+1}(t)-x_a(t)\to 1^-$ as $t\to+\infty$. By the average preservation, we argue that  $x_i(t)\to \xave(0)- \frac{N-a}N$ if $i\le a$ and $x_i(t)\to \xave(0)+ \frac{a}N$ otherwise.
As the argument can be extended to $k\ge3$ by defining $k-1$ appropriate disconnection times, we conclude that every solution converges to a point in $F$.
\qedhere
\end{enumerate}

\end{proof}

The set of equilibria $F$ in Theorem~\ref{th:DHK-2} has the
following feature: its points are such that the agent opinions either
coincide or their distance is larger than 1. Equivalently, the opinions of two agents are equal if and only if they are connected in the limit interaction graph. Following the opinion dynamics literature, we refer to such groups of agents as {\em clusters}, and to the corresponding values as {\em cluster values} or {\em cluster points}.
More formally, one can consider for a given $x\in F_n$, the map $\I\ni i\mapsto x_i\in \real$: the image of such map consists of the cluster values, and the clusters are the preimages of the cluster values. The {\em size} of a cluster is its cardinality.

\begin{remark}(Weak and strong equilibria)
According to the definition of Krasovskii equilibrium, Krasovskii
solutions which have Krasovskii equilibria as initial conditions may
leave the equilibria. 
For example, if $N=2$, $\overline x=(1,0)\in F$, there
  are two Krasovskii solutions issuing from $\overline x$:
  $x^1(t)\equiv (1,0)$ and $x^2(t)=(1/2+1/2 e^{-2t}, 1/2-1/2
  e^{-2t})$. In other words, the set~$F$ is weakly invariant but not strongly
  invariant. A subset of $F$ which is strongly invariant is  
$\mathring{F}=\Big\{x\in\real^N\;:\; \text{for every }  (i,j)\in\I\times\I, \text{ either  $x_i=x_j$ or }  |x_i-x_j|> 1\Big\}$. 
\end{remark}

\subsection{Robustness}
In~\cite{VDB-JMH-JNT:09a}, motivated by explaining simulation results about~\eqref{eq:DHK}, the authors provide a robustness analysis for the equilibria of a 
suitable weighted version of the model.
Inspired by this approach, we propose a similar definition of robustness for the equilibria of the original system~\eqref{eq:DHK}.
Loosely speaking, an equilibrium is said to be robust if the perturbation due to adding one agent is not able to make two clusters merge.
Equivalently, an equilibrium is said to be robust if, after adding one agent, there is no solution whose limit equilibrium point has a smaller number of clusters. A more formal definition can be given as follows.
\begin{definition}[Robust Equilibrium]
Let $x^*\in F$ and $x_0\in \real$, and consider a Krasovskii solution to~\eqref{eq:DHK}, $\tilde x(\cdot)$, such that $\tilde x(0)=\bar x=(x_1^*,\ldots x_N^*, x_0)$. Let $\bar{\bar x}=\lim_{t\to\infty}\tilde x(t)$. 

If, for any $x_0 \in \real$ and any complete Krasovskii solution, the number of clusters of $\bar {\bar x}$ is not smaller than the number of clusters in $x^*\!$, then $x^*$ is said to be robust.
\end{definition}

The following is our main  robustness result.
\begin{theorem}[Robustness conditions]\label{th:DHK-3}
Let $x^*\in F$ and, when considering any pair of subsequent clusters in $x^*$, denote their values as $x_A$ and $x_B$ and their sizes as $n_A$ and $n_B$, with $n_A\le n_B.$
The equilibrium $x^*$ is robust if and only if, for every pair of clusters,
either $n_A=n_B$ and $|x_A-x_B|> 2$,
or $n_A<n_B$
and 
$$|x_A-x_B|> \l(1+\frac{n_A}{n_B}\r)\l(1+\frac{1}{n_A+n_B}\r) e^{-t_{n_A,n_B}^*},$$
where the negative number $t_{n_A,n_B}^*$ is such that  
\be\label{eq:def-tab}\frac{1}{n_A+n_B} e^{-(n_A+n_B+1)t_{n_A,n_B}^*}+ \frac{n_A}{n_B}\l(1+\frac{1}{n_A+n_B}\r)e^{-t_{n_A,n_B}^*}=1.\ee
\end{theorem}

The proof of this result is postponed to the Appendix: here we briefly discuss its meaning for large populations of agents. 

\begin{corollary}[Robustness in large populations]
As $n_A\to \infty$ the necessary and sufficient condition in Theorem~\ref{th:DHK-3} degenerates into 
$$|x_A-x_B|> 1+\frac{n_A}{n_B}$$
for any pair of clusters $A$ and $B$.
\end{corollary}
\begin{proof}
From~\eqref{eq:def-tab} we deduce that
$$\frac{1}{n_A+n_B} e^{-(n_A+n_B+1)t_{n_A,n_B}^*}\le1$$
and then 
$$0\le (-t_{n_A,n_B}^*) \le \frac{\log{(n_A+n_B)}}{n_A+n_B+1}.$$
Consequently, $
\l(1+\frac{n_A}{n_B}\r)\l(1+\frac{1}{n_A+n_B}\r) e^{-t_{n_A,n_B}^*}$ is asymptotically equivalent to $1+\frac{n_A}{n_B}$ as $n_A\to\infty$.
\end{proof}

This corollary states that when clusters are very large, the simple condition $|x_A-x_B|> 1+\frac{n_A}{n_B}$ is necessary and sufficient for robustness against small perturbations. 
The interest for robust equilibria is motivated by the following intuition. 
Robust equilibria are more suitable to be limit points of ``real'' opinion dynamics system, which would be subject to various uncertainties and disturbances. Furthermore, as noted in~\cite{VDB-JMH-JNT:09a}, simulated solutions typically converge to robust equilibria.

\section{Krasovskii and \car Solutions}\label{sec:compare}

This section is devoted to compare the properties of Krasovskii solutions with those of \car solutions, and especially with those of the solutions considered in~\cite{VDB-JMH-JNT:09a}.

\subsection{Sliding Mode Solutions}
As a consequence of their definitions, the set of Krasovskii solutions may be larger than the set of solutions intended in a \car sense.
We now provide an example of a solution sliding on a discontinuity surface, proving that there are Krasovskii solutions to~\eqref{eq:DHK} which are not \car solutions. 

\begin{example}[Sliding mode]\label{ex:sliding-mode}
Let $N=3$ and consider a configuration $x$ in which $1>x_2-x_1>0$ and $x_3-x_2=1$. Then, $x$ is on a discontinuity surface due to the disconnection between agents $2$ and $3$. Then, for almost every time
$$\dot x \!\in \!\setdef{\alpha
\left[\begin{array}{c}
      x_2-x_1  \\
      1+x_1-x_2 \\
      -1 \\
   \end{array}
   \right]
\!+ (1-\alpha)
\left[
\begin{array}{c}
      x_2-x_1 \\
      x_1-x_2 \\
      0 \\
   \end{array}
   \right]}
   { \alpha\in[0,1]\!}\!.   $$
Since the normal vector to the discontinuity plane is $v_{\bot}=[0,-1,1]$, we have that
$$v_\bot \cdot \dot x=- 2 \alpha+x_2-x_1$$
is equal to zero if $\alpha=\frac12(x_2-x_1).$ Namely,  the Krasovskii
solution corresponding to such $\alpha$ does not exit 
 the discontinuity plane $x_3-x_2=1$ at time $0$, but it slides on it. The sliding solution takes into account the fact that opinions $x_3$ and $x_2$ may remain for a while at the threshold distance before reaching an equilibrium configuration.
\qed
\end{example}

It is an open question whether sliding mode solutions can be
attractive for the dynamics. However, we know
from~\cite{VDB-JMH-JNT:09a} that  a unique complete \car solution
exists for almost every initial condition. This implies that 
the set of initial conditions such that the
corresponding solutions converge to a sliding mode has measure zero, because solutions corresponding to those initial conditions would not be complete.

\subsection{Krasovskii and Proper Solutions}

In~\cite{VDB-JMH-JNT:09a} and~\cite{VDB-JMH-JNT:09b}, the authors consider a carefully defined subset of Carath\'{e}odory solutions to~\eqref{eq:DHK}: they call {\em proper solution} 
any Carath\'{e}odory solution $x(t)$ corresponding to an initial condition $x_0$ (called {\em proper initial condition}) such that 
\begin{itemize}
\item[a)] $x(t)$ is the unique Carath\'{e}odory solution to~\eqref{eq:DHK} with initial condition $x_0$ defined on $[0,+\infty )$;
\item[b)] the subset of $[0,+\infty )$ where $x(t)$ is not differentiable is at most countable, and has no accumulation points;
\item[c)] if $x_i(t)=x_j(t)$ for some $t$, then $x_i(t')=x_j(t')$ for all $t'\geq t$.
\end{itemize}

Moreover they prove that almost all $x\in \real ^N$ are proper initial conditions, and that proper solutions are contractive in the sense of Proposition~\ref{prop:DHK}, preserve the average of states and converge to clusters.
Our analysis has shown that the most significant properties of proper solutions also hold in the larger set of Krasovskii solutions.
Nevertheless, there are a few significant differences, which we detail in the following list.
\begin{itemize}
\item {\bf Existence.}
For any point $x_0$ in $\real^N$, there is a Krasovskii solution $x(t)$ such that $x(0)=x_0$. Instead, there are points, which belong to a certain set $P$ of measure zero, such that Carath\'eodory solutions starting at these points may either not exist or not be unique, so that proper solutions may not exist. 
The set $P$ includes points on the discontinuity surfaces, such that $x_i=x_j$ for some $i\neq j$. Note that Krasovskii solutions include sliding mode solutions as the one in Example~\ref{ex:sliding-mode}: such solutions belong for a positive duration of time to a discontinuity surface.
\item {\bf Uniqueness.} Proper solutions are unique by definition, whereas Krasovskii solutions are in general not unique. Note, however, that the results obtained in this note, and in particular convergence to a clustered configuration, hold for  every Krasovskii solution.
\item {\bf Regularity.} Krasovskii solutions are differentiable almost everywhere: proper solutions, by definition, are differentiable out of countable set with no accumulation point.
\item {\bf Order preservation.} Proper solutions preserve both inequalities and equalities between states, while we have remarked that Krasovskii solutions may not 
preserve equalities at discontinuities.
\item {\bf Connectivity.} 
Along proper solutions, the number of connected components in $\G(x(t))$ is nondecreasing in time. Similarly, along Krasovskii solutions the number of connected components in $\bar \G(x(t))$ is nondecreasing in time.
%

\item{\bf Robustness.}
We have defined an equilibrium to be robust if the addition of one perturbing agent does not make two clusters merge. 
In~\cite{VDB-JMH-JNT:09a}, robustness is defined for a suitable extension of the model, which provides agents with weights: an equilibrium is robust if the addition of one perturbing agent {\em of arbitrary small weight} does not make two clusters merge.
Our definition avoids defining this auxiliary weighted system, and the small-weight limit is replaced by a limit in the size of the clusters. The resulting analysis provides a necessary and sufficient condition, which takes the same simple and intuitive form as the condition in~\cite{VDB-JMH-JNT:09a}.
\end{itemize}

The main drawback of the approach taken in~\cite{VDB-JMH-JNT:09a} is difficulty in studying existence and continuation properties of those solutions. Instead, Krasovskii solutions are easier to deal with, as far as existence and continuation properties are considered. Moreover, results about proper solutions can be a posteriori obtained as particular cases of the more general results on Krasovskii solutions, which also include solutions starting at ``problematic points'' which may not admit proper solutions starting from them. For these reasons, we believe that Krasovskii solutions can be a useful tool in opinion dynamics, whenever the model involves discontinuities.


\bibliographystyle{plain}

\begin{thebibliography}{10}

\bibitem{AB-FC:99}
A.~Bacciotti and F.~Ceragioli.
\newblock Stability and stabilization of discontinuous systems and nonsmooth
  {L}iapunov functions.
\newblock {\em {ESAIM:} Control, Optimisation \& Calculus of Variations},
  4:361--376, 1999.

\bibitem{VDB-JMH-JNT:09b}
V.~D. Blondel, J.~M. Hendrickx, and J.~N. Tsitsiklis.
\newblock Existence and uniqueness of solutions for a continuous-time opinion
  dynamics model with state-dependent connectivity, July 2009.

\bibitem{VDB-JMH-JNT:09}
V.~D. Blondel, J.~M. Hendrickx, and J.~N. Tsitsiklis.
\newblock On {K}rause's multi-agent consensus model with state-dependent
  connectivity.
\newblock {\em IEEE Transactions on Automatic Control}, 54(11):2586--2597,
  2009.

\bibitem{VDB-JMH-JNT:09a}
V.~D. Blondel, J.~M. Hendrickx, and J.~N. Tsitsiklis.
\newblock Continuous-time average-preserving opinion dynamics with
  opinion-dependent communications.
\newblock {\em SIAM Journal on Control and Optimization}, 48(8):5214--5240,
  2010.

\bibitem{FB-JC-SM:09}
F.~Bullo, J.~Cort{\'e}s, and S.~Mart{\'\i}nez.
\newblock {\em Distributed Control of Robotic Networks}.
\newblock Applied Mathematics Series. Princeton University Press, 2009.

\bibitem{CC-FF-PT:08}
C.~Canuto, F.~Fagnani, and P.~Tilli.
\newblock A {E}ulerian approach to the analysis of rendez-vous algorithms.
\newblock In {\em {IFAC} {W}orld {C}ongress}, Seoul, Korea, July 2008.

\bibitem{CC-SF-VL:09}
C.~Castellano, S.~Fortunato, and V.~Loreto.
\newblock Statistical physics of social dynamics.
\newblock {\em Reviews of Modern Physics}, 81(2):591--646, 2009.

\bibitem{FC-PF:11}
F.~Ceragioli and P.~Frasca.
\newblock Continuous and discontinuous opinion dynamics with bounded
  confidence.
\newblock {\em Nonlinear Analysis: Real World Applications}, 2011.
\newblock to appear.

\bibitem{FC-PF:10}
F.~Ceragioli and P.~Frasca.
\newblock Continuous-time discontinuous equations in bounded confidence opinion
  dynamics.
\newblock In {\em {IFAC} {W}orld {C}ongress}, pages 1986--1990, Milan, Italy,
  August 2011.

\bibitem{JC:08-csm}
J.~Cort{\'e}s.
\newblock Discontinuous dynamical systems -- a tutorial on solutions, nonsmooth
  analysis, and stability.
\newblock {\em {IEEE} Control Systems Magazine}, 28(3):36--73, 2008.

\bibitem{DE-JK:10}
D.~Easley and J.~Kleinberg.
\newblock {\em Networks, {C}rowds, and {M}arkets: {R}easoning About a Highly
  Connected World}.
\newblock Cambridge University Press, 2010.

\bibitem{AFF:88}
A.~F. Filippov.
\newblock {\em Differential Equations with Discontinuous Righthand Sides},
  volume~18 of {\em Mathematics and Its Applications}.
\newblock Kluwer Academic Publishers, 1988.

\bibitem{OH:79}
O.~H{\'a}jek.
\newblock Discontinuous differential equations {I}.
\newblock {\em Journal of Differential Equations}, 32:149--170, 1979.

\bibitem{RH-UK:02}
R.~Hegselmann and U.~Krause.
\newblock Opinion dynamics and bounded confidence models, analysis, and
  simulation.
\newblock {\em Journal of Artificial Societies and Social Simulation}, 5(3),
  2002.

\bibitem{JL:07}
J.~Lorenz.
\newblock Continuous opinion dynamics under bounded confidence: A survey.
\newblock {\em International Journal of Modern Physics C}, 18(12):1819--1838,
  2007.

\bibitem{MM-ME:10}
M.~Mesbahi and M.~Egerstedt.
\newblock {\em Graph Theoretic Methods for Multiagent Networks}.
\newblock Applied Mathematics Series. Princeton University Press, 2010.

\end{thebibliography}

\appendix

\section{Proof of Theorem~\ref{th:DHK-3}}
Without loss of generality, we restrict our attention to one pair of clusters with values $x_A$, $x_B$ and sizes $n_A$, $n_B$, and a perturbing agent with value $x_0\in (x_A,x_B)$. 
We also assume  that $n_A<n_B$ and that the agents in each cluster preserve the equality between their states.\footnote{
The latter assumption is in general a restriction, as we know from the proof of Theorem~\ref{th:DHK-2} that equalities between states are not always preserved along Krasovskii solutions at discontinuities. We will show later that in this case the assumption entails no loss of generality.}
%
Thanks to this assumption, we can limit ourselves to consider the following system of three equations,
\begin{equation}\label{eq:rob-full}
\begin{cases}
\dot x_B=s(x_B-x_A) (x_0-x_B)\\
\dot x_0=-n_A s(x_0-x_A)\, (x_0-x_A) + n_B s(x_B-x_0) (x_B-x_0)\\
\dot x_A= s(x_A-x_0) (x_0-x_A).
\end{cases}
\end{equation}
Then, by defining $x=x_0-x_A$ and $y=x_B-x_0$, we are left to study the following bidimensional (discontinuous) system,
\begin{equation}\label{eq:rob-reduced}
\begin{cases}
\dot x=- (n_A+1) s(x) x + n_B s(y) y\\
\dot y= n_A s(x) x - (n_B+1) s(y) y,
\end{cases}
\end{equation}
when the initial condition is such that $\l(x(0),y(0)\r)\in (0,1)\times(0,1).$
The case of the two original clusters $A,B$ merging is equivalent to system~\eqref{eq:rob-reduced} converging to the origin. 
Then, the core of our analysis consists in studying for this system the region of attraction of the origin.

\begin{lemma}\label{lem:rob-reduced}
System~\eqref{eq:rob-reduced} has a complete Krasovskii solution converging to the origin if and only if the initial condition belongs to the region $\mathcal R$ delimited by the positive $x$ and $y$ axes, the lines $\{x=1\}$ and $\{y=1\}$ and the branch of 
the curve $$\left( - \frac{1}{n_A+n_B} e^{-(n_A+n_B+1)t} + (1+\frac{1}{n_A+n_B}) e^{-t},
\frac{1}{n_A+n_B}e^{-(n_A+n_B+1)t}  +(1+\frac{1}{n_A+n_B}) e^{-t} 
\right)$$
when $t_{n_A,n_B}^*\le t\le 0$
and $t_{n_A,n_B}^*$ is defined as in~\eqref{eq:def-tab}.
\end{lemma}

\begin{proof}
As long as $(x(t),y(t))\in (0,1)\times(0,1)$, system~\eqref{eq:rob-reduced} reduces to the linear system 
\begin{equation}\label{eq:rob-linear}
\begin{cases}
\dot x=- (n_A+1) x + n_B y\\
\dot y= n_A x - (n_B+1) y.
\end{cases}
\end{equation}
Being system~\eqref{eq:rob-linear} asymptotically stable, we argue that all solutions, which do not leave the unit square, converge to zero. Conversely, outside the unit square the system is not asymptotically stable: for instance, if $x>1$ and $y<1$, we have
$$\begin{cases}
\dot x= n_B  y\\
\dot y=  - (n_B+1)  y.
\end{cases}
$$
Then, we only need to find which solutions leave the unite square. To this goal, continuity of solutions leads us to consider the system in the limit for $x$ or $y$ approaching $1$ (from below).
When $y\to1^-$, system~\eqref{eq:rob-linear} becomes $$
\begin{cases}
\dot x=- (n_A+1) x + n_B\\
\dot y= n_A x - (n_B+1).
\end{cases}$$
Since $n_A x - (n_B+1)< n_A  - (n_B+1)< 0,$ solutions can not reach the discontinuity $\{y=1\}$.
When $x\to1^-$, system~\eqref{eq:rob-linear} becomes $$ \begin{cases}
\dot x=- (n_A+1)+ n_B y\\
\dot y= n_A - (n_B+1) y,
\end{cases}$$
implying that solutions may cross the discontinuity $\{x=1\}$ if and only if $y\ge \frac{n_A+1}{n_B}.$

Let us then consider the solution to~\eqref{eq:rob-linear}, passing at time $t=0$ by the point $\l(1,\frac{n_A+1}{n_B}\r),$ which can be written in closed form as 
$$ \begin{cases}
\tilde x(t)=\displaystyle- \frac{1}{n_A+n_B} e^{-(n_A+n_B+1)t}+\l(1+\frac{1}{n_A+n_B}\r) e^{-t}
\\
\tilde y(t)=\displaystyle  \frac{1}{n_A+n_B} e^{-(n_A+n_B+1)t}+\frac{n_A}{n_B}\l(1+\frac{1}{n_A+n_B}\r) e^{-t}.
\end{cases}$$
Note that, by definition, $\tilde y(t)=1$ if $t=t^*_{n_A,n_B}$. As $t^*_{n_A,n_B}$ is finite, the limited region $\mathcal{R}$ is well defined in the statement of the theorem. 
By uniqueness of the solutions to~\eqref{eq:rob-linear}, we argue that a solution to~\eqref{eq:rob-linear} may leave the unit square if and only its initial condition is ``above'' the solution $(\tilde x,\tilde y)$, {\it i.e.}, outside the region $\mathcal{R}.$
\end{proof}

The following corollary derives from Lemma~\ref{lem:rob-reduced} a necessary and sufficient condition for solutions to~\eqref{eq:rob-reduced} to converge to zero, in terms of the sum $x(0)+y(0).$

\begin{corollary}\label{corol:rob-reduced-bis}
Let $\bar{d}=\l(1+\frac{n_A}{n_B}\r)\l(1+\frac{1}{n_A+n_B}\r) e^{-t_{n_A,n_B}^*}$. If the initial condition is such that $x(0)+y(0)> \bar{d},$ then there exists no solution to~\eqref{eq:rob-reduced} which converges to zero. Conversely, for any $d\le \bar{d},$ one can find a pair $(x(0),y(0))$ such that $x(0)+y(0)=d$ and there is a solution originating from $(x(0),y(0))$ which converges to zero.
\end{corollary}
\begin{proof}
Let us then consider again the solution to~\eqref{eq:rob-linear} passing by $\l(1,\frac{n_A+1}{n_B}\r).$
Note that $$\tilde x(t)+\tilde y(t)=\l(1+\frac{n_A}{n_B}\r)  \l(1+\frac{1}{n_A+n_B}\r) e^{-t}.$$
This implies that 
$$\max\setdef{\tilde x(t)+\tilde y(t)}{t\in[t^*_{n_A,n_B},0]}=\tilde  x(t^*_{n_A,n_B})+\tilde y(t^*_{n_A,n_B})=\bar d$$ and 
$$\min\setdef{\tilde x(t)+\tilde y(t)}{t\in[t^*_{n_A,n_B},0]}= \tilde x(0)+\tilde y(0)=1+\frac{n_A+1}{n_B}.$$
By Lemma~\ref{lem:rob-reduced} and this remark, we deduce the following fact. Being $d>0$ fixed, we can find $(x(0),y(0))$ such that $x(0)+y(0)=d$ and $(x(0),y(0))\in \mathcal {R}$ if and only if $d<\bar d$. This proves the statement of the corollary.
\end{proof}

In order to infer Theorem~\ref{th:DHK-3} from Corollary~\ref{corol:rob-reduced-bis} we still need to remove the assumption that equalities be preserved along the evolution. To do that, we discuss the two cases in which this assumption could be restrictive, that is, when the solution starts at a discontinuity because either $x_B-x_0=1$ or $x_0-x_A=1.$
\begin{enumerate}[i)]
\item 
We assume that $x_0(0)=x_A(0)+1$ and $x_B-x_0>\frac{n_A+1}{n_B}$. (Indeed, when $x_B-x_A\le1+\frac{n_A+1}{n_B}$ we already know that there is a solution leading to cluster coalescence.)
Then, we note that $\dot x_0(0)=n_B(x_B-x_0)-n_A$ while $\dot x_a(0)\in[0,1]$ for every $a\in A.$ As $x_B-x_0>\frac{n_A+1}{n_B}$ implies $\dot x_0(0)>1$, the solution may not stay on the discontinuity, and in particular is such that $x_0(t)-x_a(t)>1$ for every $t>0$.
\item 
We assume that $x_0(0)=x_B(0)-1$.
Then, we note that $\dot x_0(0)=n_B-n_A(x_0-x_A)$ while $\dot x_b(0)\in[-1,0]$ for every $b\in B.$ As $n_B>n_A$ implies $\dot x_0(0)>0$, the solution may not stay on the discontinuity, and in particular is such that $x_B(t)-x_0(t)<1$ for every $t>0$.
\end{enumerate}

In view of this discussion, we infer that it is not restrictive for the robustness analysis  to assume that equalities be preserved, and we conclude that the state of the perturbing agent $x_0$ can be chosen in such a way to make the two clusters merge if and only if $x_B-x_A\le \bar{d}$. This proves the statement of Theorem~\ref{th:DHK-3} when $n_A<n_B$. The case of $n_A=n_B$ is much simpler and is left to the reader.

\end{document}